\documentclass[11pt]{article}
\usepackage{amsmath}
\usepackage{amsfonts}
\usepackage{amssymb}
\usepackage{hyperref}

\let\ds=\displaystyle

\newtheorem{theorem}{Theorem}

\newtheorem{lemma}{Lemma}
\newtheorem{proposition}{Proposition}
\newenvironment{proof}[1][Proof]{\noindent\textbf{#1.} }{\ \rule{0.5em}{0.5em}}

\newtheorem{definition}{Definition}
\begin{document}

\title{\textbf{$r$-extension of Dunkl operator in one variable and Bessel
functions of vector index}}
\author{
Ahmed Fitouhi \thanks{Facult\'e des sciences de Tunis, 1060 Tunis, Tunisia. E-mail :ahmed.fitouhi@fst.rnu.tn} 
\& 
Lazhar Dhaouadi \thanks{IPEIB, 7021 Zarzouna, Bizerte, Tunisia. E-mail : lazhardhaouadi@yahoo.fr} 
and
Fethi Bouzeffour \thanks{Department of mathematics, College of Sciences, King Saud University, P. O Box 2455
Riyadh 11451, Saudi Arabia. E-mail : fbouzaffour@ksu.edu.sa.}}

\date{}
\maketitle

\begin{abstract}
In this work we present an operator $D_\mu$ constructed with the
help of the cyclic group set of the $r^{{\small th}}$ roots of
unity. This operator constitute an $r$-extension of the Dunkl
operator in one variable because when $r=2$ it reduces to the
classical one and admits as eigenfunctions the Bessel functions of
vector index early deeply studied by Klyuchantsev. This paper is
argued by specific examples and contains some interesting results
 which are the prelude of  harmonic analysis related to this
operator.

\vspace{5mm} \noindent \textit{Keywords : Dunkl operator, Bessel functions,
Fourier transform, transmutation operator.} \vspace{3mm}\newline
\noindent \textit{2000 AMS Mathematics Subject Classification---Primary
33D15,47A05. }
\end{abstract}

\section{Introduction}

At the beginning of the last decade of the twentieth century C. Dunkl \cite%
{D1,D2} in a series of articles using reflection groups introduced a
differential-difference operator now commonly called Dunkl operator and
became a great center of interest and inspiration in many areas of pure and
applied mathematics. This operator has generated a rich harmonic analysis
developed by several authors. and involves a combination of Bessel functions
of index $\alpha $ as eigenfunctions . So exploiting specific properties of
these well-known special functions great analysis and an armada of
applications was born.

\bigskip

Having knowledge of the progress of this topic in many scientific areas, we
are always asked and highly intrigued by his extension in higher $r$-order
which involves Bessel functions of index vector $j_{\mu }$ which are
eigenfunctions of $\Delta _{\mu },\quad \mu =(\alpha _{1},...,\alpha _{r-1})$
a differential operator of order $r$. These last functions are one of the
generalized Bessel functions mentioned by Watson in his venerable book \cite%
{W} and greatly studied with applications by many authors (see \cite%
{B,E,F1,F2,K}) and the references therein . The cyclic group $C_{r}$ plays a
central role in the definition of our $r$-extension particulary for define
the $r$-even and the $r$-odd functions of order $l=1,...r-1$ and leads to
decompose the space of functions in direct sum of invariant subspaces with
appropriate projectors $T_{k},\quad k=1,...r-1$ useful to construct $D_{\mu }
$ the $r$-extension of the Dunkl operator having as fundamental property $%
D_{\mu }^{r}=\Delta _{\mu }$ on a particular subspace $F_{k}$. In addition
to the construction process of the operator $D_{\mu }$, we give a particular
interest in many cases especially in all the paragraphs discussed namely
representation integral, associated Riemann-Liouville transform,
transmutation and $r$-extension Dunkl transform.

Recently, someone tell us that our operator $D_\mu$ can be included
in the class of operators presented by Dunkl and Opdam \cite{D3}. We give
at the end of this work our commentary and the link between the two
buildings and we are grateful to our informant. Nerveless , in both
cases no reliable harmonic analysis concerning theses operators is
made in addition in our case explicit eigenfunctions are obtain
expressed via Bessel functions with index vector.

\section{The operator $D_{\protect\mu}$}

Throughout this paper $r$ is an integer great than $1$, $\omega =e^{\frac{%
2i\pi }{r}}$ and we put
\begin{equation*}
C_{r}=\{1,\omega ,{\omega }^{2},...{\omega }^{r-1}\}
\end{equation*}
the cyclic group of order $r$.Let $F$ be the space of complex valued
functions  on which we consider the following actions\newline

\begin{equation*}
s_{k}g(x)={\omega }^{k}g(\omega x),~~~~k=0,1,2,....
\end{equation*}
Putting $\emph{F}_{k}$ the subspace of $F$ invariant by $s_{k}$ ; namely

\begin{equation}  \label{e11}
g\in \emph{F}_k \Leftrightarrow s_kg=g.
\end{equation}

Now we introduce the collection of the projector operators defined by the
relations
\begin{equation}  \label{e2}
T_{k}={\frac{1}{r}}\sum_{n=0}^{r-1}s_{k}^{n};~~~~k=0,1,2,...
\end{equation}
which are  slightly different from those introduced in
\cite{Ri}.\bigskip

We recall that in some mathematical literature \cite{F2,K}, we often
say  the  function $T_{0}g$  the $r$-even part of $g$ and the
functions $T_{k}g$, ~~$r=1,2,...r-1$ the $r$-odd of order $k$ of
$g$.\newline Taking account of the fact that $s_{k}^{r}=id$, one can
easily show that the following properties hold :

\begin{enumerate}
\item The operators $T_{k}$ and $s_{k}$, ~~~$k=0,1,2...$ commute in the the
sense\newline
\begin{equation*}
T_ks_k=s_kT_k
\end{equation*}
\item The subspace $\emph{F}_{k}$ can be also characterized as:\newline
\begin{equation*}
g\in \emph{F}_k \Leftrightarrow T_kg=g.
\end{equation*}
\item We have
\begin{equation*}
~k\neq l~\Leftrightarrow T_{k}T_{l}=0.
\end{equation*}
\end{enumerate}

Starting of the fact that $T_k$ are projectors namely $T_k^2=T_k$
one can see easily that

\begin{equation*}
\emph{F}=F_1 \oplus ....\oplus F_{r-1}.
\end{equation*}

For more clarity and taking account of their importance and their
interference in the demonstrations we summarize here the useful
 properties :

\begin{enumerate}
\item The derivative operator $\ds\frac{d}{dx}$ maps the space $F_{k}$ into $%
F_{k+1}$.

\item The multiplication operator by $\frac{1}{x}$ satisfies
\begin{equation*}
\frac{1}{x}s_{k}=s_{k+1}\frac{1}{x}
\end{equation*}
and it maps the space $\emph{F}_{k}$ into $\emph{F}_{k+1}$. Moreover
\begin{equation*}
\frac{1}{x}T_{k}=T_{k+1}\frac{1}{x}
\end{equation*}

\item For $a$ real, let  $L_{a}$ be the operator:
\begin{equation}  \label{e1}
L_{a}(f)={x^{-a}}\frac{d}{dx}(x^{a}f)=f^{\prime }+\frac{a}{x}f.
\end{equation}

We have $x^{-b}L_{a}x^{b}=L_{a+b}$ and $L_{a}$ maps $\emph{F}_{k}$
into $\emph{F}_{k+1}$.
\end{enumerate}

\bigskip

\begin{definition}
Let $\mu =(\alpha _{0},\alpha _{1},...,\alpha _{r-1})$ be a vector of $%
\mathbb{R}^{r}$. We define the Bessel operator of order $r$ associated to
the index vector $\mu $ by
\begin{equation*}
\Delta _{\mu }=L_{a_{r-1}}\circ ....\circ L_{a{_{0}}}
\end{equation*}
where we have put
\begin{equation}  \label{e3}
a_{k}=r\alpha _{k}+k~~~k=0,1,...,r-1,
\end{equation}
and $L_{a}$ is the operator given by (\ref{e1}).
\end{definition}

\bigskip

\begin{definition}
The $r$-extension of Dunkl operator is defined by
\begin{equation*}
D_{\mu }=\frac{d}{dx}+\frac{1}{x}\sum_{r=0}^{r-1}a_{k}T_{k}
\end{equation*}
where the coefficients $a_k$ and the operators $T_k$ are given
respectively by (\ref{e3}) and (\ref{e2}).
\end{definition}

Before any thing let us justify the appellation by the following proposition

\begin{proposition}
For $f\in F_{k},\quad k=0,\ldots ,(r-1),$ we have
\begin{equation*}
D_{\mu }^{r}(f)=\Delta _{\mu }(f).
\end{equation*}
\end{proposition}

\begin{proof}
This result is first  consequence of the fact that $D_{\mu }$ maps the $%
F_{k}$ in $F_{k+1}$ because the operators $\frac{d}{dx}$ and $T_{k}$ have
the same properties. Since if $g\in F_{k}$ then $D_{\mu }g=L_{a_{k}}g$ we
deduce then $D_{\mu }g=L_{a_{k}}g\in F_{k+1}$ , hence
\begin{equation*}
D_{\mu }^{2}g=L_{a_{k+1}}L_{a_{k}}g\in F_{k+2}.
\end{equation*}%
By induction and the fact that $F_{k+r}=F_k$ we find
\begin{equation*}
\mathrm{if}~~g\in F_{k}\Rightarrow D_{\mu }g=\Delta _{\mu }g.
\end{equation*}
\end{proof}

Having defined the operators $\Delta _{\mu }$ and $D_{\mu }$, it quite
natural to seek their eigenfunctions . For this we introduce the Bessel
functions of vector index
\begin{equation*}
\mu =(\alpha _{0},\alpha _{1},...,\alpha _{r-1})\in \mathbb{R}^{r}\backslash
\mathbb{Z}_{-}^{r}
\end{equation*}
by
\begin{equation*}
j_{\mu }(x)=\sum_{n=0}^{\infty }(-1)^{n}\frac{1}{(\alpha
_{0}+1)_{n}(\alpha _{1}+1)_{n}...(\alpha
_{r-1}+1)_{n}}\frac{x^{nr}}{r^{nr}},
\end{equation*}
where $(\beta )_{n}=\frac{\Gamma (\beta +1)}{\Gamma (\beta )}$.

\bigskip

The knowledgeable reader should note that this function differs from that
studied in \cite{K} by the number of components of the index vector.\bigskip

The above series is entire and taking account of
\begin{equation*}
\Delta_\mu x^{rn}=r^r(\alpha_0+n)(\alpha_1+n)...(\alpha_{r-1}+n)x^{(n-1)r},
\end{equation*}
one can state :

\begin{proposition}
For a complex $\lambda $ we have
\begin{equation*}
\Delta _{\mu }j_{\mu }(\lambda x)=-\lambda ^{r}j_{\mu }(\lambda x);
\end{equation*}
with $j_{\mu }(0)=1$.\newline
\end{proposition}

Now we put
\begin{equation*}
\theta =e^{i\frac{\pi }{r}}
\end{equation*}
and we consider the $r$-extension function which call it also $r$-Dunkl
kernel
\begin{equation}  \label{e10}
E_{\mu }(x)=j_{\mu }(x)+\frac{1}{\theta }D_{\mu }j_{\mu }(x)+...+\frac{1}{%
\theta ^{r-2}}D_{\mu }^{r-2}j_{\mu }(x)+\frac{1}{\theta^{r-1}}D_{\mu
}^{r-1}j_{\mu }(x).
\end{equation}

\begin{proposition}
For complex $\lambda $, the function $x\mapsto E_{\mu }(\lambda x)$
is an eigenfunction for the generalized Dunkl operator $D_{\mu }$
with $\theta \lambda $ as eigenvalue:
\begin{equation*}
D_{\mu }E_{\mu }(\lambda x)=\theta \lambda E_{\mu }(\lambda x).
\end{equation*}
\end{proposition}

\begin{proof}
Indeed, to be convinced it suffices to make the following computations:

\begin{eqnarray*}
D_{\mu }E_{\mu } &=&D_{\mu }j_{\mu }+\frac{1}{\theta }D_{\mu }^{2}j_{\mu
}+...+\frac{1}{\theta ^{r-2}}D_{\mu }^{r-1}j_{\mu }+\frac{1}{\theta ^{r-1}}
D^r_{\mu }j_{\mu } \\
&=&D_{\mu }j_{\mu }+\frac{1}{\theta }D_{\mu }^{2}j_{\mu }+...+\frac{1}{
\theta ^{r-2}}D_{\mu }^{r-1}j_{\mu }+\frac{\theta ^{r}}{\theta ^{r-1}}j_{\mu
} \\
&=&\theta \left[ j_{\mu }+\frac{1}{\theta }D_{\mu }j_{\mu }+\dots +\frac{1}{%
\theta ^{r-1}}D_{\mu }^{r-1}j_{\mu }\right] \\
~~ &=&\theta E_{\mu }.
\end{eqnarray*}
so the result follows.
\end{proof}

In the reminder we must compute the action of the $r$-extension
of Dunkl operator $D_\mu$ on the Bessel functions of vector index $j_\mu$.%
\newline

This is can be deduced from the fact that $D_{\mu
}x^{nr}=L_{a_{0}}x^{nr}=r(\alpha _{0}+n)x^{nr-1}$, we obtain
\begin{equation*}
D_{\mu }j_{\mu }(x)=r\sum_{n=0}^{\infty }(-1)^{n}\frac{1}{(\alpha
_{0}+1)_{n}...(\alpha _{r-1}+1)_{n}}(\alpha _{0}+n)\frac{x^{nr-1}}{r^{nr}}.
\end{equation*}

Then we distinguish two cases :

\bigskip\textbf{Case 1 :} $\alpha _{0}\neq 0$ .\bigskip

We have
\begin{equation*}
D_{\mu }j_{\mu }(x)=\frac{r}{x}\alpha _{0}j_{\mu -1},
\end{equation*}%
remark that we have adopt the convention : for $\mu =(\alpha _{0},\alpha
_{1},...,\alpha _{r-1})$ we put
\begin{equation*}
\mu -1=(\alpha _{0}-1,\alpha _{1},...,\alpha _{r-1}).\newline
\end{equation*}

\bigskip\textbf{Case 2 :}  $\alpha _{0}=0$.\bigskip

With a slice change of computation we have
\begin{equation*}
D_{\mu }j_{\mu }(x)=-\frac{1}{(\alpha _{1}+1)...(\alpha _{r-1}+1)}\left(
\frac{x}{r}\right) ^{r-1}j_{\mu +1}.
\end{equation*}
We have also adopt the convention : for $\mu =(\alpha _{0},\alpha
_{1},...,\alpha _{r-1})$ we put
\begin{equation*}
\mu +1=(\alpha _{0},\alpha _{1}+1,...,\alpha _{r-1}+1).
\end{equation*}
As mentioned in abstract we present here three explicit examples
which illustrate the operators $\Delta _{\mu }$ and $D_{\mu }$ .

\bigskip\textbf{Example 1 : }$r=2,\omega =-1,\theta =i,\mu =(0,\alpha ).$%
\bigskip

Then
\begin{equation*}
a_{0}=0,a_{1}=2\alpha +1.
\end{equation*}

So that
\begin{equation*}
\Delta _{\mu }=L_{2\alpha +1}L_{0}=\frac{d^{2}}{dx^{2}}+\frac{2\alpha +1}{x}%
\frac{d}{dx},
\end{equation*}

which is exactly the well known Bessel operator having as
eigenfunction the normalized Bessel function\newline

\begin{equation*}
j_\alpha(x)=j_\mu(x) = \sum_{n=0}^\infty(-1)^n\frac{1}{n!(\alpha+1)_n}\frac{%
x^{2n}}{2^{2n}}.
\end{equation*}

Now since%
\begin{equation*}
T_{1}g(x)=\frac{g(x)+s_{1}g(x)}{2}=\frac{g(x)-g(-x)}{2}
\end{equation*}%
The operator $D_{\mu }$ is the classical Dunkl operator in one
variable:
\begin{equation*}
D_{\mu }=D_{\alpha }=\frac{d}{dx}+\frac{\alpha +1}{x}T_{1},
\end{equation*}%
and using the conventional notation introduced before namely
\begin{equation*}
\mu+1=(0,\alpha +1)
\end{equation*}
and the fact that $\theta =i$, lead to show that the eigenfunctions
$E_{\mu } $ coincide with the classical one.

\bigskip\textbf{Example 2 : }$r=3,\omega =e^{\frac{i2\pi }{3}},\theta =e^{%
\frac{i\pi }{3}},\mu =\left( 0,\alpha -\frac{1}{3},-\frac{2}{3}\right)
.\bigskip $

Taking account of the relation (\ref{e3}) between $a_k$ and $\alpha_k$ we
deduce that
\begin{equation*}
a_{0}=0,~~a_{1}=3\alpha,~~a_{2}=0.
\end{equation*}
So
\begin{equation*}
\Delta_{\mu }=L_{0}L_{3\alpha }L_{0}=\frac{d^{3}}{dx^{3}}-\frac{3\alpha }{x}\frac{%
d^{2}}{dx^{3}}+\frac{3\alpha }{x^{2}}\frac{d}{dx}.
\end{equation*}%
The previous operator was greatly studied in \cite{F2}. Its eigenfunction is
given by\newline
\begin{equation*}
j_{\mu }(x)=\sum_{n=0}^{\infty }(-1)^{n}\frac{1}{n!\left( \alpha +\frac{2}{3}%
\right) _{n}\left( \frac{1}{3}\right) _{n}}\frac{x^{3n}}{3^{3n}}.
\end{equation*}

The correspondent Dunkl operator is:
\begin{equation*}
D_{\mu }=\frac{d}{dx}+\frac{3\alpha }{x}T_{1}
\end{equation*}%
with
\begin{equation*}
T_{1}g(x)=\frac{g(x)+\omega g(\omega x)+\omega ^{2}g(\omega
^{2}x)}{3},
\end{equation*}

we can deduce
\begin{equation*}
D_\mu j_\mu(x)=\frac{d}{dx}j_\mu(x)=-\frac{1}{(\alpha+{\frac{2}{3}})(\frac{1%
}{3})}\left(\frac{x}{3}\right)^2j_{\mu + 1}(x),
\end{equation*}

with $\mu+1=(0,\alpha+\frac{2}{3},\frac{1}{3})$ and then

\begin{equation*}
j_{\mu+1}(x)=\sum_{n=0}^\infty(-1)^n\frac{1}{n!(\alpha+\frac{5}{3})_n(\frac{4%
}{3})_n}\frac{x^{3n}}{3^{3n}}.
\end{equation*}

From the fact that $D_{\mu }j_{\mu }\in \emph{F}_{1}$, we obtain
\begin{equation*}
D_{\mu }^{2}j_{\mu }(x)=\frac{d}{dx}D_{\mu }j_{\mu }(x)+\frac{3\alpha }{x}%
T_{1}D_{\mu }j_{\mu }(x).
\end{equation*}%
Direct computations give
\begin{equation*}
D_{\mu }^{2}j_{\mu }(x)=\frac{x^{4}}{4(3\alpha +2)(3\alpha +5)}j_{\mu
+2}(x)-xj_{\mu +1};
\end{equation*}%
with $\mu +2=(0,\mu +\frac{5}{3},\frac{4}{3})$ and
\begin{equation*}
j_{\mu +2}(x)=\sum_{n=0}^{\infty }(-1)^{n}\frac{1}{n!\left( \alpha +\frac{8}{%
3}\right) _{n}\left( \frac{7}{3}\right) _{n}}\frac{x^{3n}}{3^{3n}}.
\end{equation*}

Finally taking account of the above results we state that the eigenfunction
of the correspondent Dunkl operator is then
\begin{equation*}
E_{\mu }(x)=j_{\mu }(x)+e^{-\frac{i\pi }{3}}D_{\mu }j_{\mu }(x)+e^{-\frac{%
2i\pi }{3}}D_{\mu }^{2}j_{\mu }(x)
\end{equation*}

\begin{equation*}
E_{\mu }(x)=j_{\mu }(x)-\left[ \frac{e^{-\frac{i\pi }{3}}}{3^{2}(\alpha +%
\frac{2}{3})(\frac{1}{3})}\right] j_{\mu +1}(x)+\frac{e^{-\frac{2i\pi }{3}}}{%
4(3\alpha +2)(3\alpha +5)}x^{4}j_{\mu +2}(x).
\end{equation*}

\textbf{Remark 1 : }It is easy to see that the following commutation holds:
\begin{equation*}
L_{a}T_{k}=T_{k+1}L_{a}
\end{equation*}

and as $T_{k+r}=T_{k}$ this leads that the operators $\Delta _{\mu
}$ and $T_{k}$ commute in the sense
\begin{equation*}
\Delta _{\mu }T_{k}=T_{k}\Delta _{\mu }.
\end{equation*}

Note that if $g$ is a function such that $\Delta _{\mu }g=-g$ and
from the unique decomposition :
\begin{equation*}
g=T_{0}g+\ldots+T_{r-1}g,
\end{equation*}%
one can interpret the component $T_{k}g$ as the unique solution of
the previous equation restraint to the subspace $F_{k}$.

\bigskip \textbf{Example 3} : $\mu =(0,-\frac{1}{r},\dots ,-\frac{r-1}{r}),$
$\theta =e^{\frac{i\pi }{r}}$.\bigskip

In this situation as $\alpha _{k}=-\frac{k}{r}$ we have $a_{k}=0$ , hence $
\Delta _{\mu}=(\frac{d}{dx})^{r}$ and $D_{\mu }=\frac{d}{dx}$. It is clear
that the function $e_{\theta }(x)=e^{\theta x}$ satisfies the equation
\begin{equation*}
\Delta _{\mu }e_{\theta }=-e_{\theta }
\end{equation*}%
The components $T_{k}e^{\theta x}$, $k=0,1,...,r-1$ are called the
r-trigonometric functions \cite{F2}. We have in particular
\begin{equation*}
\cos _{r}(x)=T_{0}e^{\theta x}=\frac{1}{r}\sum_{k=0}^{r-1}e^{\theta \omega
^{k}x}=\sum_{n=0}^{\infty }(-1)^{n}\frac{1}{(1)_{n}(-\frac{1}{r}+1)_{n}...(-%
\frac{r-1}{r}+1)_{n}}\frac{x^{nr}}{r^{nr}}
\end{equation*}
This last function is the unique eigenfunction of $\Delta_{\mu}=(\frac{d}{dx}%
)^{r}$ with take the value $1$ at $x=0$. We notice that the Dunkl kernel can
be written
\begin{eqnarray*}
E_{\mu }(x) &=&\cos _{r}(x)+\frac{1}{\theta }D_{\mu }\cos _{r}(x)+...+\frac{1%
}{\theta ^{r-1}}D_{\mu }^{r-1}\cos _{r}(x) \\
&=&(T_{0}+...+T_{r-1})e^{\theta x}=e^{\theta x}.
\end{eqnarray*}

\section{Integral representations}

In this section we attempt to show that the functions $j_\mu$ and $E_\mu$
have some useful integral representations. For the first function the reader
can found tis integral representation already shown by klyuchantsev \cite{K} and
which is recalled in the proof of Theorem \ref{t1}.\bigskip

The following lemma is basic and it is a consequence of properties of Euler
functions.

\begin{lemma}
We have
\begin{equation*}
r\int_{0}^{1}(1-u^{r})^{y-1}u^{rx-1}du=\frac{\Gamma \left( x\right) \Gamma
\left( y\right) }{\Gamma \left( x+y\right) }
\end{equation*}%
provided the integral converges.
\end{lemma}

\bigskip

We deduce then the following identities

\begin{equation*}
r\int_{0}^{1}u^{nr}(1-u^{r})^{\alpha _{i}+\frac{i}{r}-1}u^{r-(i+1)}du=\frac{%
\Gamma \left( -\frac{i}{r}+1+n\right) \Gamma \left( \alpha _{i}+\frac{i}{r}%
\right) }{\Gamma \left( \alpha _{i}+1+n\right) }.
\end{equation*}

\begin{equation*}
r\frac{\Gamma \left( \alpha _{i}+1\right) }{\Gamma \left( \alpha _{i}+\frac{i%
}{r}\right) \Gamma \left( -\frac{i}{r}+1\right) }\int_{0}^{1}\frac{\Gamma
\left( -\frac{i}{r}+1\right) }{\Gamma \left( -\frac{i}{r}+1+n\right) }%
u^{nr}(1-u^{r})^{\alpha _{i}+\frac{i}{r}-1}u^{r-(i+1)}du=\frac{\Gamma \left(
\alpha _{i}+1\right) }{\Gamma \left( \alpha _{i}+1+n\right) }
\end{equation*}%
So we have
\begin{equation*}
r\frac{\Gamma \left( \alpha _{i}+1\right) }{\Gamma \left( \alpha _{i}+\frac{i%
}{r}\right) \Gamma \left( -\frac{i}{r}+1\right) }\int_{0}^{1}\frac{1}{\left(
-\frac{i}{r}+1\right) _{n}}u^{nr}(1-u^{r})^{\alpha _{i}+\frac{i}{r}%
-1}u^{r-(i+1)}du=\frac{1}{\left( \alpha _{i}+1\right) _{n}}
\end{equation*}%
and in general case

\begin{eqnarray*}
&&\int_{0}^{1}\ldots \int_{0}^{1}(-1)^{n}\frac{1}{\left( 1\right) _{n}\ldots
\left( -\frac{r-1}{r}+1\right) _{n}}\frac{\left( xu_{0}\ldots u_{r-1}\right)
^{nr}}{r^{nr}}w_{0}(u_{0})\ldots w_{r-1}(u_{r-1})du_{0}\ldots du_{r-1} \\
&=&(-1)^{n}\frac{1}{\left( \alpha _{0}+1\right) _{n}\ldots \left( \alpha
_{r-1}+1\right) _{n}}\frac{x^{nr}}{r^{nr}}.
\end{eqnarray*}
Let
\begin{equation}  \label{e4}
w_{\mu }(u)=\prod_{i=0}^{r-1}(1-u_{i}^{r})^{\alpha _{i}+\frac{i}{r}%
-1}u_{i}^{r-(i+1)}
\end{equation}
\begin{equation}  \label{e5}
c_{\mu }=\prod_{i=0}^{r-1}r\frac{\Gamma \left( \alpha _{i}+1\right) }{\Gamma
\left( \alpha _{i}+\frac{i}{r}\right) \Gamma \left( -\frac{i}{r}+1\right) }
\end{equation}
and
\begin{equation*}
u_{r}=u_{0}\ldots u_{r-1}
\end{equation*}

\begin{equation*}
du=du_{0}\ldots du_{r-1}
\end{equation*}
and using the function $\cos_r$ presented in Example 3, we have

\begin{theorem}
\label{t1} The Bessel function of vector index possess the following
integral representation
\begin{equation*}
j_{\mu }(x)=c_{\mu }\int_{[0,1]^{r}}\cos _{r}\left( xu_{r}\right) w_{\mu
}(u)du
\end{equation*}
where $c_\mu$ and $w_{\mu }$ are given respectively by (\ref{e5}) and (\ref%
{e4}).
\end{theorem}

Note that in this representation we can remove the components $u_{i}$
associated with indices $i$ such that $a_{i}=0.$ In this case the Mehler
representation takes the following form
\begin{equation*}
j_{\mu }(x)=c_{\mu ^{\prime }}\int_{[0,1]^{r^{\prime }}}\cos _{r}\left(
xu_{r^{\prime }}\right) w_{\mu ^{\prime }}(u)du
\end{equation*}%
with $r^{\prime }$ is the number of index $i$ such that $a_{i}\neq 0$ and $%
\mu ^{\prime }$ contains only the associate $\alpha _{i}$ .

\begin{theorem}
the $r$-Dunkl kernel possess the following integral representation
\begin{equation*}
E_{\mu }(x)=c_{\mu }\int_{[0,1]^{r}}\left( T_{0}+\sum_{k=1}^{r-1}\frac{1}{%
\theta ^{k}}T_{k}L_{a_{k-1}}\ldots L_{a_{0}}\right) e_{\theta }\left(
xu_{r}\right) w_{\mu }(u)du.
\end{equation*}
\end{theorem}

\begin{proof}
This is a consequence of definition of $E_{\mu }(x)$ and the identity

\begin{eqnarray*}
\frac{1}{\theta ^{k}}D_{\mu }^{k}j_{\mu }(x) &=&c_{\mu }\int_{[0,1]^{r}}%
\frac{1}{\theta ^{k}}D_{\mu }^{k}T_{0}e_{\theta }\left( xu_{r}\right) w_{\mu
}(u)du \\
&=&c_{\mu }\int_{[0,1]^{r}}\frac{1}{\theta ^{k}}L_{a_{k-1}}\ldots
L_{a_{0}}T_{0}e_{\theta }\left( xu_{r}\right) w_{\mu }(u)du \\
&=&c_{\mu }\int_{[0,1]^{r}}\frac{1}{\theta ^{k}}T_{k}L_{a_{k-1}}\ldots
L_{a_{0}}e_{\theta }\left( xu_{r}\right) w_{\mu }(u)du,
\end{eqnarray*}
which prove the result.
\end{proof}

\bigskip\textbf{Example 4 : }$r=2,w=-1,\theta =i,\ \mu =(0,\alpha )$\bigskip

Since $a_{0}=0$ then we can remove the index $i=0$ in the representation of
the function $j_{\mu }$ which gives

\begin{equation*}
j_{\mu }(x)=j_{\alpha }(x)=\frac{2}{\sqrt{\pi }}\frac{\Gamma \left( \alpha
+1\right) }{\Gamma \left( \alpha +\frac{1}{2}\right) }\int_{0}^{1}\cos
\left( xu\right) (1-u^{2})^{\alpha -\frac{1}{2}}du.
\end{equation*}%
On the other hand

\begin{eqnarray*}
E_{\mu }(x) &=&E_{\alpha }(x)=\frac{2}{\sqrt{\pi }}\frac{\Gamma \left(
\alpha +1\right) }{\Gamma \left( \alpha +\frac{1}{2}\right) }\int_{0}^{1}%
\left[ T_{0}e^{ixu}+\frac{1}{i}T_{1}\frac{d}{dx}e^{ixu}\right]
(1-u^{2})^{\alpha -\frac{1}{2}}du \\
&=&\frac{2}{\sqrt{\pi }}\frac{\Gamma \left( \alpha +1\right) }{\Gamma \left(
\alpha +\frac{1}{2}\right) }\int_{0}^{1}\left[ T_{0}e^{ixu}+uT_{1}e^{ixu}%
\right] (1-u^{2})^{\alpha -\frac{1}{2}}du \\
&=&\frac{2}{\sqrt{\pi }}\frac{\Gamma \left( \alpha +1\right) }{\Gamma \left(
\alpha +\frac{1}{2}\right) }\int_{0}^{1}\left[ T_{0}e^{ixu}+T_{1}\frac{1}{x}%
(xu)e^{ixu}\right] (1-u^{2})^{\alpha -\frac{1}{2}}du
\end{eqnarray*}%
To find the classical form we can write

\begin{eqnarray*}
E_{\alpha }(x) &=&\frac{2}{\sqrt{\pi }}\frac{\Gamma \left( \alpha +1\right)
}{\Gamma \left( \alpha +\frac{1}{2}\right) }\int_{0}^{1}\left[ \frac{%
e^{ixu}+e^{-ixu}}{2}+u\frac{e^{ixu}-e^{-ixu}}{2}\right] (1-u^{2})^{\alpha -%
\frac{1}{2}}du \\
&=&\frac{1}{\sqrt{\pi }}\frac{\Gamma \left( \alpha +1\right) }{\Gamma \left(
\alpha +\frac{1}{2}\right) }\int_{-1}^{1}e^{ixu}(1+u)(1-u^{2})^{\alpha -%
\frac{1}{2}}du.
\end{eqnarray*}

\bigskip\textbf{Example 5 :}$\mathbf{\ \ }r=3,w=e^{i\frac{2\pi }{3}%
}=j,\theta =e^{i \frac{\pi }{3}}, \mu =\left( 0,v-\frac{1}{3},-\frac{2}{3}%
\right)$.\bigskip

We have
\begin{equation*}
j_{\mu }(x)=j_{v}(x)=3\frac{\Gamma \left( v+\frac{2}{3}\right) }{\Gamma
\left( v\right) \Gamma \left( \frac{2}{3}\right) }\int_{0}^{1}\cos
_{3}\left( xu\right) (1-u^{3})^{v-1}udu.
\end{equation*}
Then
\begin{equation*}
E_{v}(x)=3\frac{\Gamma \left( v+\frac{2}{3}\right) }{\Gamma \left( v\right)
\Gamma \left( \frac{2}{3}\right) }\int_{0}^{1}\left[ T_{0}e^{\theta xu}+%
\frac{1}{\theta }T_{1}\frac{d}{dx}e^{\theta xu}+\frac{1}{\theta ^{2}}%
T_{2}\left( \frac{d}{dx}+\frac{3v}{x}\right) \frac{d}{dx}e^{\theta xu}\right]
(1-u^{3})^{v-1}udu
\end{equation*}
which can be written in the following form
\begin{multline*}
E_{v}(x)=3\frac{\Gamma \left( v+\frac{2}{3}\right) }{\Gamma \left( v\right)
\Gamma \left( \frac{2}{3}\right) }\times \\
\int_{0}^{1}\left[ T_{0}\frac{1}{x}(xu)e^{\theta xu}+T_{1}\frac{1}{x^{2}}%
(xu)^{2}e^{\theta xu}+T_{2}\frac{1}{x^{3}}(xu)^{3}e^{\theta xu}+\frac{3v}{%
\theta }T_{2}\frac{1}{x^{3}}(xu)^{2}e^{\theta xu}\right] (1-u^{3})^{v-1}du.
\end{multline*}

\section{Riemann--Liouville transform}

Considering the $r$-Riemannn Liouville operators of the form
\begin{equation*}
R_{\alpha }g(x)=\int_{0}^{1}g(xt)(1-t^{r})^{\alpha -1}dt.
\end{equation*}
The integral representation in Theorem \ref{t1} of the Bessel function of
vector index can be rewritten as follows
\begin{equation}  \label{e8}
j_{\mu }(x)=c_{\mu }\int_{[0,1]^{r}}\cos _{r}\left( xu_{r}\right)
w_{\mu}(u)du =c_{\mu }\prod_{i=0}^{r-1}\left( \frac{1}{x^{r-(i+1)}}R_{\alpha
_{i}+\frac{ i}{r}-1}x^{r-(i+1)}\right)\cos _{r}(x).
\end{equation}
Now we study the inverse of $R_\alpha$. We begin by state:

\begin{theorem}
For $k$ integer and $0<\alpha <1$ we have
\begin{equation*}
R_{k+\alpha }^{-1}g(x)=\frac{r^{2}}{\Gamma (k+1)\Gamma (\alpha )\Gamma
(1-\alpha )}x^{r-1}\left( \frac{1}{rx^{r-1}}\frac{d}{dx}\right)
^{k+1}\int_{0}^{x}g(u)(x^{r}-u^{r})^{-\alpha }u^{(k+\alpha )r}du.
\end{equation*}
\end{theorem}

\begin{proof}
The operator $R_{\alpha }$ can be take the form
\begin{equation*}
R_{\alpha }g(x)=\frac{1}{x^{1+r(\alpha -1)}}\int_{0}^{x}g(u)\left[
x^{r}-u^{r}\right] ^{\alpha -1}du
\end{equation*}
and for $0<\alpha <1$ admits as inverse :
\begin{equation*}
R_{\alpha }^{-1}g(t)=\frac{r}{\Gamma (\alpha )\Gamma (1-\alpha )}\frac{d}{dt}%
\int_{0}^{t}g(x)\left[ t^{r}-x^{r}\right] ^{-\alpha }x^{\alpha r}dx
\end{equation*}%
which is shown as follows%
\begin{eqnarray*}
R_{\alpha }^{-1}R_{\alpha }g(t) &=&\frac{r}{\Gamma (\alpha )\Gamma (1-\alpha
)}\frac{d}{dt}\int_{0}^{t}R_{\alpha }g(x)\left[ t^{r}-x^{r}\right] ^{-\alpha
}x^{\alpha r}dx \\
&=&\frac{r}{\Gamma (\alpha )\Gamma (1-\alpha )}\frac{d}{dt}\int_{0}^{t}\left[
\frac{1}{x^{1+r(\alpha -1)}}\int_{0}^{x}g(u)\left[ x^{r}-u^{r}\right]
^{\alpha -1}du\right] \left[ t^{r}-x^{r}\right] ^{-\alpha }x^{\alpha r}dx \\
&=&\frac{r}{\Gamma (\alpha )\Gamma (1-\alpha )}\frac{d}{dt}\int_{0}^{t}\left[
\int_{u}^{t}\left[ x^{r}-u^{r}\right] ^{\alpha -1}\left[ t^{r}-x^{r}\right]
^{-\alpha }x^{r-1}dx\right] g(u)du \\
&=&\frac{1}{\Gamma (\alpha )\Gamma (1-\alpha )}\frac{d}{dt}\int_{0}^{t}\left[
\int_{u^{r}}^{t^{r}}\left[ y-u^{r}\right] ^{\alpha -1}\left[ t^{r}-y\right]
^{-\alpha }dy\right] g(u)du \\
&=&\frac{d}{dt}\int_{0}^{t}g(u)du=g(t).
\end{eqnarray*}%
Therefore
\begin{equation*}
R_{\alpha }^{-1}=\frac{r}{\Gamma (\alpha )\Gamma (1-\alpha )}\frac{d}{dx}%
x^{1-\alpha r}R_{1-\alpha }x^{\alpha r}.
\end{equation*}%
For an integer $k$ we have the following relation%
\begin{equation*}
\frac{r}{k!}x^{r-1}\left( \frac{1}{rx^{r-1}}\frac{d}{dx}\right)
^{k+1}\int_{0}^{x}g(u)(x^{r}-u^{r})^{k}du=g(x)
\end{equation*}%
then
\begin{equation*}
R_{k+1}^{-1}=\frac{r}{k!}x^{r-1}\left( \frac{1}{rx^{r-1}}\frac{d}{dx}\right)
^{k+1}x^{1+kr}.
\end{equation*}%
On the other hand%
\begin{eqnarray*}
R_{k+\alpha }g(x) &=&\frac{1}{x^{1+(k+\alpha -1)r}}%
\int_{0}^{x}(x^{r}-u^{r})^{k+\alpha -1}du \\
&=&\frac{1}{x^{1+(k+\alpha -1)r}}%
\int_{0}^{x}(x^{r}-u^{r})^{k}(x^{r}-u^{r})^{\alpha -1}du \\
&=&\frac{1}{x^{kr}}\frac{1}{x^{1+(\alpha -1)r}}\int_{0}^{x}\left[ \frac{d}{du%
}\int_{0}^{u}(x^{r}-s^{r})^{k}ds\right] (x^{r}-u^{r})^{\alpha -1}du
\end{eqnarray*}%
then%
\begin{equation*}
R_{k+\alpha }=\frac{1}{x^{kr}}R_{\alpha }\frac{d}{dx}x^{1+rk}R_{k+1}.
\end{equation*}%
So we have
\begin{equation*}
R_{k+\alpha }^{-1}=R_{k+1}^{-1}\frac{1}{x^{1+rk}}\left( \frac{d}{dx}\right)
^{-1}R_{\alpha }^{-1}x^{kr}
\end{equation*}%
\begin{equation*}
R_{k+\alpha }^{-1}=\frac{r^{2}}{\Gamma (k+1)\Gamma (\alpha )\Gamma (1-\alpha
)}x^{r-1}\left( \frac{1}{rx^{r-1}}\frac{d}{dx}\right) ^{k+1}x^{1-\alpha
r}R_{1-\alpha }x^{(k+\alpha )r}.
\end{equation*}%
The result is then established.
\end{proof}

\section{Hilbertian structure}

We equipped the space $F$ of complex valued functions by the hermitian
scalar product given by:
\begin{equation}  \label{e6}
\left\langle f,g\right\rangle _{a}=\int_{0}^{\infty }\left[
\sum_{m=0}^{r-1}f(w^{m}t)\overline{g(w^{m}t)}\right] t^{a}dt
\end{equation}
where $a$ is a suitable positive real number.\bigskip

We need to list some properties of the resulting hermitian structure. We
begin by showing that the projectors $T_{i}$ given by (\ref{e2}) are then
symmetric . Indeed

\begin{eqnarray*}
\left\langle f,T_{i}g\right\rangle _{a} &=&\int_{0}^{\infty }\left[
\sum_{m=0}^{r-1}f(w^{m}t)\overline{T_{i}g(w^{m}t)}\right] t^{a}dt \\
&=&\frac{1}{r}\int_{0}^{\infty }\left[ \sum_{m=0}^{r-1}\sum_{k=0}^{r-1}%
\overline{w}^{ik}f(w^{m}t)\overline{g(w^{m+k}t)}\right] t^{a}dt \\
&=&\frac{1}{r}\int_{0}^{\infty }\left[ \sum_{k=0}^{r-1}\sum_{m^{\prime
}=0}^{r-1}w^{-ik}f(w^{m^{\prime }-k}t)\overline{g(w^{m^{\prime }}t)}\right]
t^{a}dt \\
&=&\frac{1}{r}\int_{0}^{\infty }\left[ \sum_{m^{\prime
}=0}^{r-1}\sum_{k^{\prime }=0}^{r-1}w^{ik^{\prime }}f(w^{m^{\prime
}+k^{\prime }}t)\overline{g(w^{m^{\prime }}t)}\right] t^{a}dt \\
&=&\left\langle T_{i}f,g\right\rangle _{a}.
\end{eqnarray*}%
As a direct consequence we notice that if $i\neq j$ and $f\in F_{i},g\in
F_{j}$ then we have
\begin{equation*}
\left\langle f,g\right\rangle =\left\langle T_{i}f,T_{j}g\right\rangle
=\left\langle T_{j}T_{i}f,g\right\rangle =0.
\end{equation*}%
One can also verify the following identities
\begin{equation*}
\left\langle f,\frac{1}{x}g\right\rangle _{a}=\left\langle \frac{1}{%
\overline{x}}f,g\right\rangle _{a}
\end{equation*}
and
\begin{equation*}
\left\langle f,xg\right\rangle _{a}=\left\langle \overline{x}%
f,g\right\rangle _{a}.
\end{equation*}

Hence we have:

\begin{proposition}
Let $f$ and $g$ be two complex valued functions such as
\begin{equation*}
\lim_{x\rightarrow 0,\infty }\left[ f(x)\overline{g(x)}x^{a}\right] \quad{%
is~~finite}.
\end{equation*}
Then
\begin{equation*}
\left\langle \frac{d}{dx}f,g\right\rangle _{a} =-\left\langle f,\left( \frac{%
d}{dx}+\frac{a}{\overline{x}}\right) g\right\rangle _{a}.
\end{equation*}
\end{proposition}

\begin{proof}
We have
\begin{eqnarray*}
\left\langle \frac{d}{dx}f,g\right\rangle _{a}
&=&\sum_{m=0}^{r-1}\int_{0}^{\infty }\frac{df}{dx}(w^{m}x)\overline{g(w^{m}x)%
}x^{a}dx \\
&=&\sum_{m=0}^{r-1}\int_{0}^{\infty }\frac{1}{w^{m}}\frac{d}{dx}f(w^{m}x)%
\overline{g(w^{m}x)}x^{a}dx \\
&=&\left( \sum_{m=0}^{r-1}\frac{1}{w^{m}}\right) \left\{ \lim_{x\rightarrow
\infty }\left[ f(x)\overline{g(x)}x^{a}\right] -\lim_{x\rightarrow 0}\left[
f(x)\overline{g(x)}x^{a}\right] \right\} \\
&&-\sum_{m=0}^{r-1}\int_{0}^{\infty }f(w^{m}x)\overline{\left[ \frac{dg}{dx}%
(w^{m}x)+\frac{a}{\overline{w^{m}x}}g(w^{m}x)\right] }x^{a}dx \\
&=&-\left\langle f,\left( \frac{d}{dx}+\frac{a}{\overline{x}}\right)
g\right\rangle _{a}.
\end{eqnarray*}%
This is true because we have

\begin{equation*}
\sum_{m=0}^{r-1}\frac{1}{w^{m}}=0.
\end{equation*}
\end{proof}

Now we are able to determine the adjoint of the $r$-extension of the Riemann
Liouville operator and those related to the $r$- extension of Dunkl operator
:

\begin{proposition}
For $k$ integer and $0<\alpha<1$ the adjoint of the Riemann--Liouville
operator is given by
\begin{equation*}
R_{\alpha }^{\ast }g(u)=\int_{1}^{\infty }g(ut)\left[ t^{r}-1\right]
^{\alpha -1}t^{a-1-r(\alpha-1)}dt
\end{equation*}
and
\begin{eqnarray*}
R_{k+\alpha }^{\ast -1}g(\lambda ) &=&(-1)^{k+1}\frac{r^{1-k}}{\Gamma
(k+1)\Gamma (\alpha )\Gamma (1-\alpha )}\lambda ^{(k+1+\alpha )r-1} \\
&\times& \int_{1}^{\infty }\left( \frac{d}{dx}\frac{1}{x^{r-1}}+\frac{a}{%
x^{r}}\right) ^{k+1}g(\lambda x)\left( x^{r}-1\right) ^{\alpha
-1}x^{a-2-r(\alpha -2)}dx.
\end{eqnarray*}
\end{proposition}

\begin{proof}
In fact we have
\begin{eqnarray*}
\left\langle R_{\alpha }f,g\right\rangle _{a}
&=&\sum_{m=0}^{r-1}\int_{0}^{\infty }R_{\alpha }f(w^{m}x)\overline{g(w^{m}x)}%
x^{a}dx \\
&=&\sum_{m=0}^{r-1}\int_{0}^{\infty }\left[ \int_{0}^{x}f(w^{m}u)\left[
x^{r}-u^{r}\right] ^{\alpha -1}du\right] \overline{g(w^{m}x)}x^{a-1-r(\alpha
-1)}dx \\
&=&\sum_{m=0}^{r-1}\int_{0}^{\infty }f(w^{m}u)\left[ \int_{u}^{\infty }%
\overline{g(w^{m}x)}\left[ x^{r}-u^{r}\right] ^{\alpha -1}x^{a-1-r(\alpha
-1)}dx\right] du \\
&=&\sum_{m=0}^{r-1}\int_{0}^{\infty }f(w^{m}u)\overline{R_{\alpha }^{\ast
}g(w^{m}u)}du \\
&=&\left\langle f,R_{\alpha }^{\ast }g\right\rangle _{a}.
\end{eqnarray*}%
Therefore%
\begin{eqnarray*}
R_{\alpha }^{\ast }g(u) &=&u^{-a}\int_{u}^{\infty }g(x)\left[ x^{r}-u^{r}%
\right] ^{\alpha -1}x^{a-1-r(\alpha -1)}dx \\
&=&\int_{1}^{\infty }g(ut)\left[ t^{r}-1\right] ^{\alpha -1}t^{a-1-r(\alpha
-1)}dt.
\end{eqnarray*}%
Since we have
\begin{equation*}
R_{k+\alpha }^{-1}=\frac{r^{2}}{\Gamma (k+1)\Gamma (\alpha )\Gamma (1-\alpha
)}x^{r-1}\left( \frac{1}{rx^{r-1}}\frac{d}{dx}\right) ^{k+1}x^{1-\alpha
r}R_{1-\alpha }x^{(k+\alpha )r}
\end{equation*}%
then
\begin{equation*}
\left( R_{k+\alpha }^{-1}\right) ^{\ast }=R_{k+\alpha }^{\ast -1}=(-1)^{k+1}%
\frac{r^{1-k}}{\Gamma (k+1)\Gamma (\alpha )\Gamma (1-\alpha )}\overline{x}%
^{(k+\alpha )r}R_{1-\alpha }^{\ast }\overline{x}^{1-\alpha r}\left( \left(
\frac{d}{dx}+\frac{a}{\overline{x}}\right) \frac{1}{\overline{x}^{r-1}}%
\right) ^{k+1}\overline{x}^{r-1}
\end{equation*}%
which leads to the result.
\end{proof}

\begin{proposition}
The corresponding adjoint of the $r$-extension Dunkl operator namely
\begin{equation*}
D_{\mu }=\frac{d}{dx}+\frac{1}{x}\sum_{k=0}^{r-1}a_{k}T_{k}
\end{equation*}
is given by
\begin{equation*}
D_{\mu }^{\ast }=-\left( \frac{d}{dx}+\frac{1}{\overline{x}}%
\sum_{k=0}^{r-1}(a-a_{k})T_{k+1}\right),
\end{equation*}
where $a$ is the real taking place in the definition of the inner product (%
\ref{e6}).
\end{proposition}

\begin{proof}
We performs the following calculation
\begin{eqnarray*}
&&\left\langle \left( \frac{d}{dx}+\frac{1}{x}\sum_{k=0}^{r-1}a_{k}T_{k}%
\right) f,g\right\rangle _{a}=\left\langle \frac{d}{dx}f,g\right\rangle
_{a}+\left\langle \frac{1}{x}\sum_{k=0}^{r-1}a_{k}T_{k}f,g\right\rangle _{a}
\\
&=&-\left\langle f,\frac{d}{dx}g\right\rangle _{a}-\left\langle f,\frac{a}{%
\overline{x}}g\right\rangle _{a}+\left\langle f,\sum_{k=0}^{r-1}a_{k}T_{k}%
\frac{1}{\overline{x}}g\right\rangle _{a} \\
&=&-\left\langle f,\frac{d}{dx}g\right\rangle _{a}-\left\langle f,\frac{1}{%
\overline{x}}\sum_{k=0}^{r-1}aT_{k+1}g\right\rangle _{a}+\left\langle f,%
\frac{1}{\overline{x}}\sum_{k=0}^{r-1}a_{k}T_{k+1}g\right\rangle _{a} \\
&=&-\left\langle f,\left( \frac{d}{dx}+\frac{1}{\overline{x}}%
\sum_{k=0}^{r-1}\left( a-a_{k}\right) T_{k+1}\right) g\right\rangle _{a}.
\end{eqnarray*}%
This proves the result.
\end{proof}

\textbf{Example 6 :} $r=2,w=-1,\theta =i,\ \mu =(0,\alpha )$\bigskip

We choose $a=2\alpha +1$ then we get%
\begin{equation*}
\left\langle f,g\right\rangle =\int_{0}^{\infty }\left[ f(t)g(t)+f(-t)g(-t)%
\right] t^{2\alpha +1}dt=\int_{-\infty }^{\infty }f(t)g(t)\left\vert
t\right\vert ^{2\alpha +1}dt.
\end{equation*}%
On the other hand the Dunkl operator is given by
\begin{equation*}
D_{\alpha }=\frac{d}{dx}+\frac{2\alpha +1}{x}T_{1}
\end{equation*}%
which implies%
\begin{equation*}
D_{\alpha }^{\ast }=-\left( \frac{d}{dx}+\frac{2\alpha +1}{x}T_{1}\right)
=-D_{\alpha }.
\end{equation*}

\bigskip

\textbf{Example 7 :} $r=3,w=e^{i\frac{2\pi }{3}}=j,\theta =e^{i\frac{\pi }{3}%
},\ \mu =\left( 0,v-\frac{1}{3},-\frac{2}{3}\right) $\newline

We choose $a=3v.$ The Dunkl operator is given by%
\begin{equation*}
D_{v}=\frac{d}{dx}+\frac{3v}{x}T_{1}
\end{equation*}%
which implies%
\begin{equation*}
D_{v}^{\ast }=-\left( \frac{d}{dx}+\frac{3v}{x}T_{1}\right) =-D_{v}.
\end{equation*}

We note that in general we have $D_{\mu }^{\ast }\neq -D_{\mu }$; the
equality depends of a suitable choice of the real $a$.

\section{Transmutation operator $V_{\protect\mu }$}

An interesting topics is to seek an operator $V_\mu$ (see \cite{De,Ro} for
the classical one case $r=2$) which transforms $e^{\theta x}$ into $E_{\mu
}(x).$ To make this section self containing we recall some properties shown
early .
\begin{equation*}
R_{\alpha }T_{i}=T_{i}R_{\alpha },\text{ \ }T_{i}\frac{1}{x}=\frac{1}{x}%
T_{i-1},\text{ \ }T_{i}x=xT_{i+1},\text{ \ }T_{i+r}=T_{i},\text{ \ }%
T_{i}^{2}=T_{i},\text{ \ }T_{i}T_{j}=0\text{ if }i\neq j
\end{equation*}

\begin{theorem}
\label{t2} The transmutation kernel $V_\mu$ has the following form
\begin{multline}  \label{e7}
V_{\mu }=c_{\mu }T_{0}\prod_{i=0}^{r-1}\left( \frac{1}{x^{r-(i+1)}}R_{\alpha
_{i}+\frac{i}{r}-1}x^{r-(i+1)}\right) \\
+c_{\mu }\sum_{k=1}^{r-1}\sum_{j=0}^{k}\frac{P_{j}}{\theta ^{j}}T_{k}\frac{1%
}{x^{k}}\prod_{i=0}^{r-1}\left( \frac{1}{x^{r-(i+1)}}R_{\alpha _{i}+\frac{i}{%
r}-1}x^{r-(i+1)}\right) x^{k-j}.
\end{multline}
where
\begin{equation*}
P_{k-s}=\frac{1}{s!}\sum_{j=0}^{s}(-1)^{s-j}C_{s-j}^{j}%
\prod_{i=0}^{k-1}(a_{i}+i+j).
\end{equation*}
\end{theorem}

\begin{proof}
We start with the representation integral (\ref{e8}) and since we can write
\begin{equation*}
L_{a_{k-1}}\ldots L_{a_{0}}=\sum_{j=0}^{k}P_{j}\frac{1}{x^{j}}\left( \frac{d%
}{dx}\right) ^{k-j}.
\end{equation*}
The constants $P_{j}$ will be explained later. So we have
\begin{eqnarray*}
\left( \sum_{k=0}^{r-1}\frac{1}{\theta ^{k}}T_{k}L_{a_{k-1}}\ldots
L_{a_{0}}\right) e_{\theta }\left( xu_{r}\right) &=&\sum_{k=0}^{r-1}\frac{1}{%
\theta ^{k}}T_{k}\left( \sum_{j=0}^{k}P_{j}\frac{1}{x^{j}}\left( \frac{d}{dx}%
\right) ^{k-j}e_{\theta }\left( xu_{r}\right) \right) \\
&=&\sum_{k=0}^{r-1}T_{k}\left( \sum_{j=0}^{k}\frac{P_{j}}{\theta ^{j}}\frac{1%
}{x^{j}}u_{r}^{k-j}\right) e_{\theta }\left( xu_{r}\right).
\end{eqnarray*}
Hence%
\begin{eqnarray*}
E_{\mu }(x) &=&c_{\mu }\int_{[0,1]^{r}}\left( T_{0}+\sum_{k=1}^{r-1}\frac{1}{%
\theta ^{k}}T_{k}L_{a_{k-1}}\ldots L_{a_{0}}\right) e_{\theta }\left(
xu_{r}\right) w_{\mu }(u)du \\
&=&c_{\mu }\int_{[0,1]^{r}}\left( T_{0}+\sum_{k=1}^{r-1}T_{k}\left(
\sum_{j=0}^{k}\frac{P_{j}}{\theta ^{j}}\frac{1}{x^{j}}u_{r}^{k-j}\right)
\right) e_{\theta }\left( xu_{r}\right) w_{\mu }(u)du.
\end{eqnarray*}%
The transmutation operator $V_{\mu }$ is written as
\begin{eqnarray*}
V_{\mu }g(x) &=&c_{\mu }\int_{[0,1]^{r}}\left(
T_{0}+\sum_{k=1}^{r-1}T_{k}\left( \sum_{j=0}^{k}\frac{P_{j}}{\theta ^{j}}%
\frac{1}{x^{j}}u_{r}^{k-j}\right) \right) g\left( xu_{r}\right) w_{\mu }(u)du
\\
&=&c_{\mu }\left[ T_{0}\int_{[0,1]^{r}}g\left( xu_{r}\right) w_{\mu
}(u)du+\sum_{k=1}^{r-1}\sum_{j=0}^{k}T_{k}\frac{P_{j}}{\theta ^{j}}\frac{1}{%
x^{j}}\int_{[0,1]^{r}}g\left( xu_{r}\right) w_{\mu }(u)u_{r}^{k-j}du\right].
\end{eqnarray*}%
Note that%
\begin{equation*}
\int_{\lbrack 0,1]^{r}}g\left( xu_{r}\right) w_{\mu
}(u)du=\prod_{i=0}^{r-1}\left( \frac{1}{x^{r-(i+1)}}R_{\alpha _{i}+\frac{i}{r%
}-1}x^{r-(i+1)}\right) g(x),
\end{equation*}%
then%
\begin{eqnarray*}
\int_{\lbrack 0,1]^{r}}g\left( xu_{r}\right) w_{\mu }(u)u_{r}^{k-j}du &=&%
\frac{1}{x^{k-j}}\int_{[0,1]^{r}}g\left( xu_{r}\right) \left( xu_{r}\right)
^{k-j}w_{\mu }(u)du \\
&=&\frac{1}{x^{k-j}}\prod_{i=0}^{r-1}\left( \frac{1}{x^{r-(i+1)}}R_{\alpha
_{i}+\frac{i}{r}-1}x^{r-(i+1)}\right) x^{k-j}g(x).
\end{eqnarray*}%
Finally%
\begin{eqnarray*}
V_{\mu } &=&c_{\mu }T_{0}\prod_{i=0}^{r-1}\left( \frac{1}{x^{r-(i+1)}}%
R_{\alpha _{i}+\frac{i}{r}-1}x^{r-(i+1)}\right) \\
&&+c_{\mu }\sum_{k=1}^{r-1}\sum_{j=0}^{k}\frac{P_{j}}{\theta ^{j}}T_{k}\frac{%
1}{x^{k}}\prod_{i=0}^{r-1}\left( \frac{1}{x^{r-(i+1)}}R_{\alpha _{i}+\frac{i%
}{r}-1}x^{r-(i+1)}\right) x^{k-j}.
\end{eqnarray*}%
In this representation we can remove the components associated with indices $%
\ i$ such that $a_{i}=0$.\bigskip

To explicate the constants $P_{j}$ we recall that
\begin{equation*}
L_{a}=x^{-a}\frac{d}{dx}x^{a}=\frac{d}{dx}+\frac{a}{x}.
\end{equation*}
We check that
\begin{equation*}
\frac{1}{x^{k}}\prod_{i=0}^{k-1}\left( x\frac{d}{dx}+a_{i}+i\right)
=L_{a_{k-1}}\ldots L_{a_{0}.}
\end{equation*}
The use of the modified identity proven by Klushantsev \cite{K}
\begin{equation*}
\frac{1}{x^{k}}\prod_{.j=0}^{k-1}\left( x\frac{d}{dx}+a_{j}+j\right)
=\sum_{j=0}^{k}P_{j}\frac{1}{x^{j}}\left( \frac{d}{dx}\right) ^{k-j},
\end{equation*}%
where
\begin{equation*}
P_{k-s}=\frac{1}{s!}\sum_{j=0}^{s}(-1)^{s-j}C_{s-j}^{j}%
\prod_{i=0}^{k-1}(a_{i}+i+j)
\end{equation*}
leads to the result.
\end{proof}

\textbf{Remark 2 :} Thanks to relation (\ref{e7}) of the Theorem \ref{t2} ,
we can compute the inverse of the operator $V_\mu$ but being given the
complicity of writing we just give $V_\mu^{-1}$ in the case of the following
example.

\bigskip\textbf{Example 8 : }$r=2,w=-1,\theta =i,\ \mu =(0,\alpha )$\bigskip

We have%
\begin{equation*}
R_{\alpha +\frac{1}{2}}g(x)=\int_{0}^{1}g(xu)(1-u^{2})^{\alpha -\frac{1}{2}%
}du,
\end{equation*}
then%
\begin{equation*}
c_{\mu }=c_{\alpha }=2\frac{\Gamma \left( \alpha +1\right) }{\Gamma \left(
\alpha +\frac{1}{2}\right) \Gamma \left( \frac{1}{2}\right) }.
\end{equation*}
The operator $V_\mu=V_\alpha$ is then
\begin{equation*}
V_{\alpha }=c_{\alpha }\left[ T_{0}R_{\alpha +\frac{1}{2}}+T_{1}\frac{1}{x}%
R_{\alpha +\frac{1}{2}}x\right]
\end{equation*}
and we obtain
\begin{equation*}
V_{\alpha }^{-1}=\frac{1}{c_{\alpha }}\left[ R_{\alpha +\frac{1}{2}%
}^{-1}T_{0}+\frac{1}{x}R_{\alpha +\frac{1}{2}}^{-1}xT_{1}\right].
\end{equation*}
That can be justified as follows
\begin{eqnarray*}
V_{\alpha }^{-1}V_{\alpha } &=&R_{\alpha +\frac{1}{2}}^{-1}T_{0}R_{\alpha +%
\frac{1}{2}}+\frac{1}{x}R_{\alpha +\frac{1}{2}}^{-1}xT_{1}\frac{1}{x}%
R_{\alpha +\frac{1}{2}}x \\
&=&R_{\alpha +\frac{1}{2}}^{-1}R_{\alpha +\frac{1}{2}}T_{0}+\frac{1}{x}%
R_{\alpha +\frac{1}{2}}^{-1}x\frac{1}{x}T_{0}R_{\alpha +\frac{1}{2}}x \\
&=&T_{0}+\frac{1}{x}R_{\alpha +\frac{1}{2}}^{-1}R_{\alpha +\frac{1}{2}}T_{0}x
\\
&=&T_{0}+T_{1}=id.
\end{eqnarray*}
On the other hand the adjoint of $V_\alpha$ take the form:%
\begin{eqnarray*}
V_{\alpha }^{\ast } &=&c_{\alpha }\left[ R_{\alpha +\frac{1}{2}}^{\ast
}T_{0}^{\ast }+xR_{\alpha +\frac{1}{2}}^{\ast }\frac{1}{x}T_{1}^{\ast }%
\right] \\
&=&c_{\alpha }\left[ R_{\alpha +\frac{1}{2}}^{\ast }T_{0}+xR_{\alpha +\frac{1%
}{2}}^{\ast }\frac{1}{x}T_{1}\right]
\end{eqnarray*}%
where%
\begin{equation*}
R_{\alpha +\frac{1}{2}}^{\ast }g(u)=\int_{1}^{\infty }g(ut)\left[ t^{2}-1%
\right] ^{\alpha -\frac{1}{2}}tdt,
\end{equation*}%
and
\begin{equation*}
V_{\alpha }^{\ast -1}=\frac{1}{c_{\alpha }}\left[ T_{0}R_{\alpha +\frac{1}{2}%
}^{\ast -1}+T_{1}xR_{\alpha +\frac{1}{2}}^{\ast -1}\frac{1}{x}\right].
\end{equation*}%
\textbf{Example 9 :} $r=3,w=e^{i\frac{2\pi }{3}}=j,\theta =e^{i\frac{\pi }{3}%
},\ \mu =\left( 0,v-\frac{1}{3},-\frac{2}{3}\right) $\bigskip

We have

\begin{equation*}
R_{v}g(x)=\int_{0}^{1}g(xu)(1-u^{3})^{v-1}du
\end{equation*}%
\begin{equation*}
c_{\mu }=c_{v}=3\frac{\Gamma \left( v+\frac{2}{3}\right) }{\Gamma \left(
v\right) \Gamma \left( \frac{2}{3}\right) }.
\end{equation*}%
The operator $V_\mu=V_v$ is given by
\begin{equation*}
V_{v}=c_{v}\left[ T_{0}\frac{1}{x}R_{v}x+T_{1}\frac{1}{x^{2}}R_{v}x^{2}+T_{2}%
\frac{1}{x^{3}}R_{v}x^{3}+\frac{3v}{\theta }T_{2}\frac{1}{x^{3}}R_{v}x^{2}%
\right]
\end{equation*}%
then its inverse is given by
\begin{equation*}
V_{v}^{-1}=\frac{1}{c_{v}}\left[ \frac{1}{x}R_{v}^{-1}xT_{0}+\frac{1}{x^{2}}%
R_{v}^{-1}x^{2}T_{1}+\frac{1}{x^{3}}R_{v}^{-1}x^{3}T_{2}-\frac{3v}{\theta }%
\frac{1}{x^{3}}R_{v}^{-1}x^{2}T_{1}\right] .
\end{equation*}%
Since
\begin{eqnarray*}
V_{v}^{-1}V_{v} &=&\frac{1}{x}R_{v}^{-1}xT_{0}\frac{1}{x}R_{v}x+\frac{1}{%
x^{2}}R_{v}^{-1}x^{2}T_{1}\frac{1}{x^{2}}R_{v}x^{2}+\frac{1}{x^{3}}%
R_{v}^{-1}x^{3}T_{2}\frac{1}{x^{3}}R_{v}x^{3} \\
&&+\frac{3v}{\theta }\frac{1}{x^{3}}R_{v}^{-1}x^{3}T_{2}\frac{1}{x^{3}}%
R_{v}x^{2}-\frac{3v}{\theta }\frac{1}{x^{3}}R_{v}^{-1}x^{2}T_{1}\frac{1}{%
x^{2}}R_{v}x^{2} \\
&=&T_{0}+T_{1}+T_{2}+\frac{3v}{\theta }\frac{1}{x}T_{1}-\frac{3v}{\theta }%
\frac{1}{x}T_{1}=T_{0}+T_{1}+T_{2}=id
\end{eqnarray*}%
On the other hand%
\begin{equation*}
V_{v}^{\ast }=c_{v}\left[ \overline{x}R_{v}^{\ast }\frac{1}{\overline{x}}%
T_{0}+\overline{x}^{2}R_{v}^{\ast }\frac{1}{\overline{x}^{2}}T_{1}+\overline{%
x}^{3}R_{v}^{\ast }\frac{1}{\overline{x}^{3}}T_{2}+\frac{3v}{\theta }%
\overline{x}^{2}R_{v}^{\ast }\frac{1}{\overline{x}^{3}}T_{2}\right]
\end{equation*}%
where%
\begin{equation*}
R_{v}^{\ast }g(u)=\int_{1}^{\infty }g(ut)\left[ t^{3}-1\right] ^{v-1}t^{2}dt.
\end{equation*}

\section{The operators $D_{\protect\mu }$ and $\frac{d}{dx}$}

In this section, we tackle the crucial subject concerning the research of
functional spaces on which the following transmutation relation is valid
\begin{equation*}
D_{\mu }V_{\mu }=V_{\mu }\frac{d}{dx}
\end{equation*}
The first idea that comes to mind is to verify that
\begin{equation}  \label{e9}
D_{\mu }V_{\mu }x^{n}=V_{\mu }\frac{d}{dx}x^{n},\text{ \ \ }\forall n\in
\mathbb{N}
\end{equation}
and when this last fact is true then the transmutation act on the space of
entire function .
\begin{equation*}
D_{\mu }V_{\mu }g(x)=V_{\mu }\frac{d}{dx}g(x),\quad g(x)=\sum_{n=0}^{\infty
}a_{n}x^{n}.
\end{equation*}%
In fact we can permute each of the following operators%
\begin{equation*}
\frac{d}{dx},R_{\alpha },T_{k},x^{k},\frac{1}{x^{k}}
\end{equation*}%
with the infinite sum $\sum_{n=0}^{\infty }.$

\bigskip

We will give two examples and we will constat that in the first, formula (%
\ref{e9}) is true but for the second it is false .

\bigskip\textbf{Example 10 }: $r=2,w=-1,\theta =i,\ \mu =(0,\alpha )$\bigskip

We will prove that formula (\ref{e9}) is true in this case. For this we use
the following result
\begin{equation*}
R_{\alpha }x^{n}=\left( \int_{0}^{1}(1-u^{r})^{\alpha -1}u^{n}du\right)
x^{n}=\frac{1}{r}\frac{\Gamma \left( \frac{n+1}{r}\right) \Gamma \left(
\alpha \right) }{\Gamma \left( \alpha +\frac{n+1}{r}\right) }%
x^{n}=l_{n}^{\alpha }x^{n},
\end{equation*}%
then

\begin{eqnarray*}
D_{\alpha }V_{\alpha }x^{2n} &=&c_{\alpha }\left( \frac{d}{dx}+\frac{2\alpha
+1}{x}T_{1}\right) \left( T_{0}R_{\alpha +\frac{1}{2}}+T_{1}\frac{1}{x}%
R_{\alpha +\frac{1}{2}}x\right) x^{2n} \\
&=&c_{\alpha }\left( \frac{d}{dx}+\frac{2\alpha +1}{x}T_{1}\right)
l_{2n}^{\alpha +\frac{1}{2}}x^{2n}=c_{\alpha }l_{2n}^{\alpha +\frac{1}{2}%
}(2n)x^{2n-1}
\end{eqnarray*}

\begin{eqnarray*}
V_{\alpha }\frac{d}{dx}x^{2n} &=&c_{\alpha }\left( T_{0}R_{\alpha +\frac{1}{2%
}}+T_{1}\frac{1}{x}R_{\alpha +\frac{1}{2}}x\right) (2n)x^{2n-1} \\
&=&c_{\alpha }l_{2n}^{\alpha +\frac{1}{2}}(2n)x^{2n-1}.
\end{eqnarray*}%
On the other hand

\begin{eqnarray*}
D_{\alpha }V_{\alpha }x^{2n+1} &=&c_{\alpha }\left( \frac{d}{dx}+\frac{%
2\alpha +1}{x}T_{1}\right) \left( T_{0}R_{\alpha +\frac{1}{2}}+T_{1}\frac{1}{%
x}R_{\alpha +\frac{1}{2}}x\right) x^{2n+1} \\
&=&c_{\alpha }\left( \frac{d}{dx}+\frac{2\alpha +1}{x}T_{1}\right)
l_{2n+2}^{\alpha +\frac{1}{2}}x^{2n+1}=c_{\alpha }l_{2n+2}^{\alpha +\frac{1}{%
2}}\left[ (2n+1)+\left( 2\alpha +1\right) \right] x^{2n},
\end{eqnarray*}
and
\begin{eqnarray*}
V_{\alpha }\frac{d}{dx}x^{2n+1} &=&c_{\alpha }\left( T_{0}R_{\alpha +\frac{1%
}{2}}+T_{1}\frac{1}{x}R_{\alpha +\frac{1}{2}}x\right) (2n+1)x^{2n} \\
&=&c_{\alpha }l_{2n}^{\alpha +\frac{1}{2}}(2n+1)x^{2n}.
\end{eqnarray*}
To show equality we use the identity

\begin{equation*}
l_{2n}^{\alpha +\frac{1}{2}}(2n+1)=l_{2n+2}^{\alpha +\frac{1}{2}}\left[
(2n+1)+\left( 2\alpha +1\right) \right] .
\end{equation*}

\bigskip\textbf{Example 11 }: $r=3,w=e^{i\frac{2\pi }{3}}=j,\theta =e^{i%
\frac{\pi }{3}},\ \mu =\left( 0,v-\frac{1}{3},-\frac{2}{3}\right) $\bigskip

The transmutation operator is given by
\begin{equation*}
V_{v}=c_{v}\left[ T_{0}\frac{1}{x}R_{v}x+T_{1}\frac{1}{x^{2}}R_{v}x^{2}+T_{2}%
\frac{1}{x^{3}}R_{v}x^{3}+\frac{3v}{\theta }T_{2}\frac{1}{x^{3}}R_{v}x^{2}%
\right].
\end{equation*}%
The 3-extension of Dunkl operator takes the following form%
\begin{equation*}
D_{v}=\frac{d}{dx}+\frac{3v}{x}T_{1}
\end{equation*}%
We check easily that%
\begin{equation*}
\left( \frac{d}{dx}+\frac{3v}{x}T_{1}\right) V_{v}x^{3n}\neq V_{v}\frac{d}{dx%
}x^{3n}.
\end{equation*}

So the spaces of entire function seems not suitable for transmutation for
all $r$ except for the case $r=2$ .

\bigskip

In the following statement we show that the transmutation is true over the
following suitable functional space.

\begin{theorem}
Let $g$ be a continuously differentiable function on an interval $\left[ -%
\frac{T}{2},\frac{T}{2}\right] $ such that
\begin{equation}
\sum_{n=-\infty }^{\infty }\left\vert c_{n}(g)\right\vert e^{\pi \left\vert n%
\mathrm{Im}\left( w^{k}\right) \right\vert }<\infty ,\quad \forall k=0\ldots
r-1  \label{e13}
\end{equation}%
then we have
\begin{equation*}
D_{\mu }V_{\mu }g(x)=V_{\mu }\frac{d}{dx}g(x).
\end{equation*}
\end{theorem}

\begin{proof}
Since we have%
\begin{equation*}
D_{\mu }V_{\mu }e^{\theta \mu x}=V_{\mu }\frac{d}{dx}e^{\theta \mu x},\text{
\ \ }\forall \mu \in \mathbb{C}\Rightarrow D_{\mu }V_{\mu }e^{i\lambda
x}=V_{\mu }\frac{d}{dx}e^{i\lambda x},\text{ \ \ }\forall \lambda \in
\mathbb{C}.
\end{equation*}
Let $g$ be a continuously differentiable function on an interval $\left[
-\frac{T}{2},\frac{T}{2}\right] $ then we have%
\begin{equation*}
g(x)=\sum_{n=-\infty }^{\infty }c_{n}(g)e^{\frac{2i\pi }{T}nx},\quad\forall
t\in \left[ -\frac{T}{2},\frac{T}{2}\right].
\end{equation*}%
The coefficients $c_{n}(g)$ so called the Fourier coefficients of $g$,
defined by the formula%
\begin{equation*}
c_{n}(g)=\frac{1}{T}\int_{-\frac{T}{2}}^{\frac{T}{2}}g(t)e^{\frac{2i\pi }{T}%
nt}dt.
\end{equation*}%
We have
\begin{equation*}
\sum_{n=-\infty }^{\infty }\left\vert c_{n}(g)\right\vert <\infty .
\end{equation*}%
The action of the operator $T_{k}$ at the function $g$ shows in the Fourier
series a terms of the form%
\begin{equation*}
e^{\frac{2i\pi }{T}nw^{k}t},\text{ \ \ }w=e^{\frac{2i\pi }{r}},\text{ \ \ \ }%
0\leq k\leq r-1.
\end{equation*}%
As
\begin{equation*}
\left\vert e^{\frac{2i\pi }{T}nw^{k}t}\right\vert \leq e^{\pi \left\vert n%
\mathrm{Im}\left( w^{k}\right) \right\vert .}
\end{equation*}
So if we impose the condition of normal convergence%
\begin{equation*}
\sum_{n=-\infty }^{\infty }\left\vert c_{n}(g)\right\vert e^{\pi \left\vert n%
\mathrm{Im}\left( w^{k}\right) \right\vert }<\infty,\quad\forall k=0\ldots
r-1
\end{equation*}%
we say that
\begin{equation*}
D_{\mu }V_{\mu }g(x)=V_{\mu }\frac{d}{dx}g(x).
\end{equation*}
\end{proof}

Now we understand why this relationship is verified for $x^{n\text{ }}$ in
the cas $r=2$ because $w=-1$ and then $\mathrm{Im}\left( w^{k}\right) =0%
\mathrm{.}$

\section{r-extension of Dunkl transform}

Before anything let us introduce the integral transform of Laplace type
given for $\theta = i\frac{\pi}{r}$ by :
\begin{equation*}
\mathcal{L}_{\theta }g(\lambda )=\int_{0}^{\infty }e^{\theta t\lambda}g(t)dt.
\end{equation*}

\begin{proposition}
The inversion formula of $\mathcal{L}_{\theta }$ is given by%
\begin{equation*}
\mathcal{L}_{\theta }^{-1}g(x)=\frac{1}{2\pi i\overline{\theta }}%
\lim_{T\rightarrow \infty }\int_{-c\overline{\theta }-i\overline{\theta }%
T}^{-c\overline{\theta }+i\overline{\theta }T}e^{-\theta xs}g(s)ds
\end{equation*}
which is valid for any function of exponential type $\alpha<c$.
\end{proposition}

\begin{proof}
To prove this formula, given a function $g$ of exponential type $\alpha<c$.
Then there exists $M>0$ such that
\begin{equation*}
\left\vert g(t)\right\vert \leq Me^{\alpha t},\quad\forall t\in \mathbb{R}.
\end{equation*}
If $s=-c\overline{\theta }+iy\overline{\theta }$ where $c>\alpha $ and $y>0$
we get%
\begin{equation*}
\mathcal{L}_{\theta }g(s)=\int_{0}^{\infty }e^{\theta t(-c\overline{\theta }%
+iy\overline{\theta })}g(t)dt=\int_{0}^{\infty
}e^{iyt}e^{-ct}g(t)dt.
\end{equation*}%
Therefore%
\begin{equation*}
\left\vert \mathcal{L}_{\theta }g(s)\right\vert \leq \int_{0}^{\infty
}e^{-ct}\left\vert g(t)\right\vert dt\leq \int_{0}^{\infty }e^{(\alpha
-c)t}dt<\infty .
\end{equation*}
So we have%
\begin{eqnarray*}
\mathcal{L}_{\theta }^{-1}\mathcal{L}_{\theta }g(x) &=&\frac{1}{2\pi i%
\overline{\theta }}\lim_{T\rightarrow \infty }\int_{-c\overline{\theta }-i%
\overline{\theta }T}^{-c\overline{\theta }+i\overline{\theta }T}e^{-\theta
xs}\mathcal{L}_{\theta }g(s)ds \\
&=&e^{cx}\left[ \frac{1}{2\pi }\lim_{T\rightarrow \infty
}\int_{-T}^{T}e^{-iyx}\left( \int_{0}^{\infty }e^{iyt}e^{-ct}g(t)dt\right) dy%
\right] \\
&=&e^{cx}e^{-cx}g(x)=g(x),\quad\forall x\in \mathbb{R}.
\end{eqnarray*}
which prove the result
\end{proof}

\begin{definition}
For $a>0$ , we define the $r$-extension of the Dunkl transform associated
with the vector $\mu =(\alpha _{0},\alpha _{1}\ldots ,\alpha _{r-1})$ \ as
follows%
\begin{equation*}
\mathcal{F}_{\mu }g(\lambda )=\left\langle g,E_{\mu }(\lambda
x)\right\rangle _{a}=\int_{0}^{\infty }\left[ \sum_{m=0}^{r-1}g(w^{m}t)%
\overline{E_\mu (w^{m}\lambda t)}\right] t^{a}dt
\end{equation*}
where $E_{\mu }$ denote the $r$-Dunkl kernel (\ref{e10}).
\end{definition}

Taking account of the fact that  $E_{\mu}(\lambda x)=V_{\mu }e^{\theta
\lambda x}$ then we can write
\begin{equation*}
\mathcal{F}_{\mu }g(\lambda )=\left\langle g,E_{\mu }(\lambda
x)\right\rangle _{a}=\left\langle g,V_{\mu }e^{\theta \lambda
x}\right\rangle _{a}=\left\langle V_{\mu }^{\ast }g,e^{\theta \lambda
x}\right\rangle _{a}.
\end{equation*}
The integral transform associated with $\mu =(0,-\frac{1}{r},\ldots ,-\frac{%
r-1}{r})$ is given by
\begin{equation*}
\mathcal{F}_{r}g(\lambda )=\left\langle g,e^{\theta \lambda x}\right\rangle
_{0}.
\end{equation*}
This operator coincide with the Laplace transform and Fourier transform
respectively for $r=1$ and $r=2$.\bigskip

We deduce that
\begin{equation*}
\mathcal{F}_{\mu }g(\lambda )=\left\langle V_{\mu }^{\ast }g,e^{\theta
\lambda x}\right\rangle _{a}=\left\langle \left\vert x\right\vert ^{a}V_{\mu
}^{\ast }g,e^{\theta \lambda x}\right\rangle _{0}.
\end{equation*}
Therefore
\begin{equation*}
\mathcal{F}_{\mu }=\mathcal{F}_{r}\left\vert x\right\vert ^{a}V_{\mu }^{\ast
}.
\end{equation*}

\begin{proposition}
Let $g$ be a function of exponential type belongs in $F_{r-k}$ the subspace
defined by (\ref{e11}) then we have
\begin{equation*}
\mathcal{F}_{\mu }^{-1}g(\lambda )=\frac{1}{r}V_{\mu }^{\ast -1}\left\vert
x\right\vert ^{-a}\mathcal{L}_{\theta }^{-1}g(\lambda ).
\end{equation*}
\end{proposition}

\begin{proof}
We write the transformation $\mathcal{F}_{r}$ as follows
\begin{equation*}
\mathcal{F}_{r}g(\lambda )=\left\langle g,e^{\theta \lambda x}\right\rangle
_{0}=\int_{0}^{\infty }\left( \sum_{m=0}^{r-1}g(w^{m}t)e^{w^{m}\theta
t\lambda }\right) dt
\end{equation*}%
then
\begin{eqnarray*}
\mathcal{F}_{r}T_{k}g(\lambda ) &=&\int_{0}^{\infty }\left(
\sum_{m=0}^{r-1}T_{k}g(w^{m}t)e^{w^{m}\theta t\lambda }\right) dt \\
&=&\int_{0}^{\infty }\left( \sum_{m=0}^{r-1}w^{(r-k)m}e^{w^{m}\theta
t\lambda }\right) T_{k}g(t)dt \\
&=&rT_{r-k}\mathcal{L}_{\theta }T_{k}g(\lambda ).
\end{eqnarray*}
Furthermore we have

\begin{equation*}
\mathcal{F}_{\mu }T_{k}=\mathcal{F}_{r}\left\vert x\right\vert ^{a}V_{\mu
}^{\ast }T_{k}=\mathcal{F}_{r}T_{k}\left\vert x\right\vert ^{a}V_{\mu
}^{\ast }=rT_{r-k}\mathcal{L}_{\theta }T_{k}\left\vert x\right\vert
^{a}V_{\mu }^{\ast }=rT_{r-k}\mathcal{L}_{\theta }\left\vert x\right\vert
^{a}V_{\mu }^{\ast }T_{k},
\end{equation*}%
then%
\begin{equation*}
\mathcal{F}_{\mu }=\sum_{k=0}^{r-1}\mathcal{F}_{\mu
}T_{k}=r\sum_{k=0}^{r-1}T_{r-k}\mathcal{L}_{\theta }\left\vert x\right\vert
^{a}V_{\mu }^{\ast }T_{k},
\end{equation*}%
which implies
\begin{equation*}
T_{m}\mathcal{F}_{\mu }=rT_{m}\sum_{k=0}^{r-1}T_{r-k}\mathcal{L}_{\theta
}\left\vert x\right\vert ^{a}V_{\mu }^{\ast }T_{k}=rT_{m}\mathcal{L}_{\theta
}\left\vert x\right\vert ^{a}V_{\mu }^{\ast }T_{r-m}.
\end{equation*}
Notice that
\begin{equation*}
\mathcal{F}_{\mu }:F_{k}\rightarrow F_{r-k}.
\end{equation*}%
In the space $F_{k}$ we have the following equality%
\begin{equation*}
\mathcal{F}_{\mu }=rT_{r-k}\mathcal{L}_{\theta }\left\vert x\right\vert
^{a}V_{\mu }^{\ast }.
\end{equation*}
This leads to the result.
\end{proof}

\begin{proposition}
If $D_{\mu }^{\ast }=-D_{\mu }$ then we have
\begin{equation*}
\mathcal{F}_{\mu }D_{\mu }g(\lambda )=-\theta \lambda \mathcal{F}_{\mu
}g(\lambda ).
\end{equation*}
\end{proposition}

\begin{proof}
In fact
\begin{eqnarray*}
\mathcal{F}_{\mu }D_{\mu }^{\ast }g(\lambda ) &=&\left\langle D_{\mu }^{\ast
}g,E_{\mu }(\lambda x)\right\rangle _{a}=\left\langle D_{\mu }^{\ast
}g,V_{\mu }e^{\theta \lambda x}\right\rangle _{a} \\
&=&\left\langle g,D_{\mu }V_{\mu }e^{\theta \lambda x}\right\rangle
_{a}=\left\langle g,V_{\mu }\frac{d}{dx}e^{\theta \lambda x}\right\rangle
_{a} \\
&=&\theta \lambda \left\langle g,E_{\mu }(\lambda x)\right\rangle
_{a}=\theta \lambda F_{\mu }g(\lambda ).
\end{eqnarray*}
In the case $D_{\mu }^{\ast }=-D_{\mu }$ we obtain the result.\bigskip

Note that the function $x\mapsto e^{\theta \lambda x}$ is a continuously
differentiable function which satisfies (\ref{e13}).
\end{proof}

\bigskip\textbf{Epilogue}\bigskip

We just built an $r$-extension of Dunkl operator focusing on examples. This
approach is very positive and encourages researchers to determine adequate
harmonic analysis and especially look for applications.\newline
In forthcoming papers we will study in great detail the associated heat and
wave equations.

\bigskip\textbf{Appendix}\bigskip

The one-dimensional specialization of the Dunkl-Opdam
operators defined in \cite[p.20]{D3} is a particular case of this introduced
in our paper. We begin by recalling that for a fixed $r=1,2,\ldots $ the
complex reflection group $W$ of type $G\left( r,1,N\right) $ is generated by
the $N\times N$ permutation matrices with the nonzero entries being powers
of $\omega =e^{2i\pi /r},$ an $r^{th}$ root of unity, and by the complex
reflection $\tau _{i}$ defined by
\begin{equation*}
x\tau _{i}=\left( x_{1},\ldots ,\overset{i}{\omega x_{i}},\ldots \right)
,\quad 1\leq i\leq N
\end{equation*}%
An element $w$ of the groups $W$ acting on a complex valued function $f$ as
follows%
\begin{equation*}
wf(x)=f(wx).
\end{equation*}
Given a list of complex numbers $\kappa =\left( \kappa _{0},\ldots ,\kappa
_{r-1}\right) .$ The Dunkl-Opdam operators for complex reflection groups $W$
 defined by
\begin{equation*}
T_{i}\left( \kappa \right) =\frac{\partial }{\partial x_{i}}+\kappa
_{0}\sum_{j\neq i}\sum_{s=0}^{r-1}\frac{1-\tau _{i}^{-s}\left( i,j\right)
\tau _{i}^{s}}{x_{i}-\omega ^{s}x_{j}}+\sum_{t=1}^{r-1}\kappa
_{t}\sum_{s=0}^{r-1}\frac{\omega ^{-st}\tau _{i}^{s}}{x_{i}},
\end{equation*}%
where $\left( i,j\right) $ is a transposition
\begin{equation*}
x\left( i,j\right) =\left( x_{1},\ldots ,\overset{i}{x_{j}},\ldots ,\overset{%
j}{x_{i}},\ldots \right) .
\end{equation*}%
The one-dimensional specialization of this operators is then
\begin{eqnarray*}
T\left( \kappa \right)  &=&\frac{d}{dx}+\frac{1}{x}\sum_{t=1}^{r-1}\kappa
_{t}\left( \sum_{s=0}^{r-1}\omega ^{-st}\tau ^{s}\right)  \\
&=&\frac{d}{dx}+\frac{1}{x}\sum_{s=0}^{r-1}\left( \sum_{t=1}^{r-1}\kappa
_{t}\omega ^{-st}\right) \tau ^{s},
\end{eqnarray*}
where $\tau$  is a complex reflection given by
\begin{equation*}
\tau f(x)=f(\omega x).
\end{equation*}
The operators introduced in our paper can be written in the following form
\begin{eqnarray*}
D_{\mu } &=&\frac{d}{dx}+\frac{1}{x}\sum_{t=0}^{r-1}a_{t}T_{t} \\
&=&\frac{d}{dx}+\frac{1}{x}\sum_{t=0}^{r-1}a_{t}\left( {\frac{1}{r}}%
\sum_{s=0}^{r-1}s_{t}^{s}\right)  \\
&=&\frac{d}{dx}+\frac{1}{x}\sum_{s=0}^{r-1}\left( {\frac{1}{r}}%
\sum_{t=0}^{r-1}a_{t}\omega ^{st}\right) \tau ^{s}.
\end{eqnarray*}%
If we have
\begin{equation*}
{\frac{1}{r}}\sum_{t=0}^{r-1}a_{t}\omega ^{st}=\sum_{t=1}^{r-1}\kappa
_{t}\omega ^{-st},\quad s=0\ldots r-1
\end{equation*}%
then we obtain
\begin{equation*}
D_{\mu }=T\left( \kappa \right).
\end{equation*}%
Hence, to reduce an operator $D_{\mu }$ to a fixed operator $T\left( \kappa
\right) $ it is necessary to solve a linear system of  $r$ indeterminate $%
\left( a_{0},\ldots ,a_{r-1}\right) $ and $r$ equations $\left( s=0,\ldots
,r-1\right) $. There's  one and unique solution because
\begin{equation*}
\left( \omega ^{st}\right) _{0\leq t,s\leq n}\in GL_{n}(\mathbb{C)}
\end{equation*}
Conversely, if we want to reduce an operator $T(\kappa) $ to a fixed operator $D_{\mu }$ then it is necessary to solve a
linear system of $r-1$ indeterminate $\left( \kappa _{1},\ldots ,\kappa
_{r-1}\right) $ and $r$ equations $\left( s=0,\ldots ,r-1\right) $. In
general there's no solution. This prove that the operator introduced in our
paper is a generalization of the Dunkl-Opdam operators in one dimension.


\begin{thebibliography}{99}
\bibitem{B} Y. Ben Cheikh; Decomposition of the Bessel functions with
respect to cyclic group of order $n$, Matematiche, Volume 52, Article
365-378 (1997).

\bibitem{C} F. M. Cholewinski and J. A. Reneke; The generalized Airy
diffusion equation, Electron. J. Differential Equations, Volume 87, Article
1-64 (2003).

\bibitem{D1} C. F. Dunkl; Differential-difference operators associated to
reflection groups, Trans. Amer. Math. Soc, Volume 311, Article 167-183
(1989).

\bibitem{D2} C. F. Dunkl; Integral kernels with reflection group invariance,
Canad. J. Math, Volume 43, Number 6, Article 1213-1227 (1991).

\bibitem{D3} C. F. Dunkl and E. M. Opdam; Dunkl operators for complex
reflection groups, Proc. London Math. Soc, Volume 86, Article 70-108 (2003).

\bibitem{De} M. De Jeu; The Dunkl transform, Invent. Math, Volume 113,
Number 1, Article 147-162 (1993).

\bibitem{E} A. Erdeley; Higher Transcendental Functions, McGraw-Hill. New
York, Volume 2,3 (1953).

\bibitem{F1} A. Fitouhi, N. H. Mahmoud and S. A. Ould Ahmed Mahmoud;
Polynomial expansions for solutions of higher-order Bessel heat equations,
J. Math. Anal. Appl, Volume 206, Article 155-167 (1997).

\bibitem{F2} A. Fitouhi, M. S. Ben Hammouda and W. Binous; On a third
singular differential operator and transmutation, F. J. M. S, Volume 21,
Issue 3, Article 303-329 (2006).

\bibitem{K} M. I. Klyuchanstsev; Singular differential operators with $r-1$
parameters and Bessel function of vector index, Siberian Mathematical
Journal, Volume 24, Number 3, Article 353-367 (1983).

\bibitem{Ri} P. F. Ricci; Le funzioni pseudo-iperboliche
epseudo-trogonometriche, Publ. Insti. mat. Appl. Ing. Univ. Roma, Volume 12,
Article 27-49 (1978).

\bibitem{Ro} M. Rosler; Dunkl operators: theory and applications, in
orthogonal polynomials and special functions, Springer Lect. Notes Math,
Volume 1817, Article 93-135 (2003).

\bibitem{W} G. N. Watson; A treatise on the theory of Bessel Functions,
Cambridge Univ. Press. London. New York, 2nd ed (1944).
\end{thebibliography}
\end{document}